\documentclass[11pt]{amsart}

\usepackage{amssymb}
\usepackage{bm}
\usepackage[centertags]{amsmath}
\usepackage{amsfonts}
\usepackage{amsthm}
\usepackage{mathrsfs}
\usepackage[usenames]{xcolor}
\usepackage[normalem]{ulem}
\usepackage{comment}

\theoremstyle{definition}
\newtheorem{thm}{Theorem}[section]
\newtheorem{defn}[thm]{Definition}
\newtheorem{lem}[thm]{Lemma}
\newtheorem{prop}[thm]{Proposition}
\newtheorem{cor}[thm]{Corollary}
\newtheorem{rem}[thm]{Remark}

\newtheorem*{defn*}{Definition}

\newtheorem*{thm*}{Theorem}

\newtheorem*{cor*}{Corollary}

\newtheorem*{prp*}{Proposition}

\newtheorem{problem}{Problem}

\newtheorem*{problem*}{Problem}

\newtheorem*{ntt}{Notation}

\newtheorem{thmA}{Theorem}

\newtheorem{problemA}[thmA]{Problem}

\newcommand{\al}{\alpha}

\newcommand{\la}{\lambda}

\newcommand{\e}{\varepsilon}

\newcommand{\N}{\mathbb{N}}

\newcommand{\X}{\mathfrak{X}^{1/2}_{0,1}}
\newcommand{\W}{W_{0,1}}

\newcommand{\iii}[1]{{\left\vert\kern-0.25ex\left\vert\kern-0.25ex\left\vert #1 
    \right\vert\kern-0.25ex\right\vert\kern-0.25ex\right\vert}}

\newcommand\restr[2]{{% we make the whole thing an ordinary symbol
  \left.\kern-\nulldelimiterspace % automatically resize the bar with \right
  #1 % the function
  \vphantom{\big|} % pretend it's a little taller at normal size
  \right|_{#2} % this is the delimiter
  }}

\DeclareMathOperator{\supp}{supp}

\DeclareMathOperator{\ran}{ran}

\long\def\symbolfootnote[#1]#2{\begingroup%
\def\thefootnote{\fnsymbol{footnote}}\footnote[#1]{#2}\endgroup}

\allowdisplaybreaks

\begin{document}

\title[asymptotically symmetric without spreading models]{Asymptotically symmetric spaces with hereditarily non-unique spreading models}
%\dedicatory{In memory of Edward Odell}

\author[D. Kutzarova]{Denka Kutzarova}
\address{Department of Mathematics University of Illinois at Urbana-Champaign Urbana, IL 61801, USA and Institute of Mathematics and Informatics, Bulgarian
Academy of Sciences, Sofia, Bulgaria.}
\email{denka@illinois.edu}

\author[P. Motakis]{Pavlos Motakis}
\address{Department of Mathematics, University of Illinois at Urbana-Champaign, Urbana, IL 61801, U.S.A.}
\email{pmotakis@illinois.edu}

\thanks{The second named author
  was  supported by the National Science Foundation under Grant Numbers
  DMS-1600600 and DMS-1912897.}

\begin{abstract}
We examine a variant of a Banach space $\mathfrak{X}_{0,1}$ defined by Argyros, Beanland, and the second named author that has the property that it admits precisely two spreading models in every infinite dimensional subspace. We prove that this space is asymptotically symmetric and thus it provides a negative answer to a problem of Junge, the first. named author, and Odell.
\end{abstract}

\maketitle

%\setcounter{tocdepth}{1}
%\tableofcontents
\section{Introduction}
The notion of an asymptotically symmetric Banach space was introduced in \cite{JKO}. A Banach space $X$ is asymptotically symmetric if the asymptotic behavior of arrays of bounded sequences in $X$ behaves well under permutations in the following way: there exists $C\geq 1$ so that if $(x_j^{(1)})_j,\ldots,(x_j^{(n)})_j$, are bounded sequences in $X$ and $\sigma$ is a permutation of $\{1,\ldots,n\}$ then whenever the iterated limits
\[L_1 = \lim_{j_1\to\infty}\cdots\lim_{j_n\to\infty}\Bigg\|\sum_{i=1}^nx^{(i)}_{j_i}\Bigg\| \text{ and } L_2 = \lim_{j_1\to\infty}\cdots\lim_{j_n\to\infty}\Bigg\|\sum_{i=1}^nx^{(\sigma(i))}_{j_i}\Bigg\|\]
both exist then $L_1\leq CL_2$. This definition is isomorphic and it is a relaxation of the notion of stable spaces from \cite{KM}, in which $L_1 = L_2$. As it was observed in \cite{JKO}, this is indeed a relaxation of stability: Tsirelson space from \cite{T} is asymptotically symmetric but does not admit an equivalent stable norm. This is because stable spaces must always contain a subspace $X$ isomorphic to some $\ell_p$, $1\leq p<\infty$ (see \cite{KM}) and the space $T$ is an asymptotic-$\ell_1$ space that contains no such subspace $X$. Naturally one may wonder whether asymptotically symmetric spaces must have subspaces that are asymptotic-$\ell_p$ spaces. 

\begin{problemA}[\cite{JKO}]
\label{main question}
Let $X$ be an asymptotically symmetric Banach space. Does $X$ contain an infinite dimensional asymptotic-$\ell_p$ or asymptotic-$c_0$ subspace?
\end{problemA}

This problem belongs to a general class of questions that ask whether a property concerning the asymptotic behavior of arrays of sequences (or any other structure for that matter) in a Banach space $X$ can provide more information about other aspects of the asymptotic behavior of $X$ (see, e.g., \cite{FOSZ} and \cite{AM3}). The property of being an asymptotic-$\ell_p$ or $c_0$ space concerns the asymptotic behavior of a Banach space $X$ as a whole and not only that of arrays of sequences. It was first introduced in \cite{MT} for Banach spaces with bases and it was later generalized in \cite{MMT} to all Banach spaces. This definition is based on the notion of a two-player game between a player (S) and a player (V). Given a Banach space $X$, for every $n\in\mathbb{N}$, the game can be played in a version consisting of $n$ successive rounds. In each round $1\leq k\leq n$ player (S) first chooses a finite codimensional subspace $Y_k$ of $X$ and player (V) then chooses a norm-one vector $y_k$ in $Y_k$. For $1\leq p\leq \infty$ the space $X$ is called an asymptotic-$\ell_p$  (or asymptotic-$c_0$ if $p=\infty$) space if there exists $C\geq 1$ so that for every $n$-version of the game player (S) has a winning strategy to force the sequence $(y_k)_{k=1}^n$ (which was picked by player (V)) to be $C$--equivalent to the unit vector basis of $\ell_p^n$.

To solve Problem \ref{main question} in the negative direction we consider a slight variation $\X$ of a reflexive Banach space $\mathfrak{X}_{0,1}$ defined in \cite{ABM}. The spaces defined in that paper have hereditarily heterogeneous spreading model structure. Recall, if $(x_j)_j$ is a sequence in a Banach space and $(e_i)_i$ is a sequence in a seminormed space we say that $(x_j)_j$ generates $(e_i)_i$ as a spreading model if for every $n\in\N$ and scalars $a_1,\ldots,a_n$
\[\lim_{j_1\to\infty}\cdots\lim_{j_n\to\infty}\Bigg\|\sum_{i=1}^na_ix_{j_i}\Bigg\| = \Bigg\|\sum_{i=1}^na_ie_i\Bigg\|.\]
The above definition is from \cite{BS}. It is almost evident that if  $X$ is an asymptotic-$\ell_p$ (or asymptotic-$c_0$) space then every spreading model generated by a weakly null sequence in $X$ must be equivalent to the unit vector basis of $\ell_p$ (or $c_0$). The characterizing property of the space $\mathfrak{X}_{0,1}$ is that all spreading models generated by normalized weakly null sequences in $\mathfrak{X}_{0,1}$ are either equivalent to the unit vector basis of $\ell_1$ or of $c_0$ and both of these sequences appear as spreading models in every subspace of $\mathfrak{X}_{0,1}$. Thus, $\mathfrak{X}_{0,1}$ has no asymptotic-$\ell_p$ or asymptotic-$c_0$ subspace. We slightly modify the definition of the space $\mathfrak{X}_{0,1}$ to obtain a space $\X$ that retains the aforementioned property and it is additionally asymptotically symmetric.

The space $\X$ is defined with a norming set via the method of saturation under constraints with very fast growing averages. This is a Tsirelson-type method that was first used by Odell and Schlumprecht in \cite{OS1} and \cite{OS2}. It was later refined in \cite{ABM}, \cite{AM1} and others. In these papers a central tool in this method was introduced, namely the $\al$-index. This index is assigned to a block sequence in the ambient Banach space and it can obtain either one of two values: zero or not zero. This tool is useful in deciding what spreading model is generated by a given block sequence. We refine this tool by defining the quantified $\al$-index of a block sequence in $\X$. This refinement allows us to provide better estimates that eventually yield that the space $\X$ is asymptotically symmetric. In addition to the above, the quantified $\al$-index allows us to characterize the asymptotic models of the space $\X$. Recall that a infinite array of sequences $(x_j^{(i)})_j$, $i\in\N$, in a Banach space $X$ generates a sequence $(e_i)_i$ in a seminormed space as an asymptotic model if for every $n\in\N$ and scalars $a_1,\ldots,a_n$
\[\lim_{j_1\to\infty}\cdots\lim_{j_n\to\infty}\Bigg\|\sum_{i=1}^na_ix^{(i)}_{j_i}\Bigg\| = \Bigg\|\sum_{i=1}^na_ie_i\Bigg\|.\]
This definition was introduced in \cite{HO}. The definition of asymptotically symmetric spaces can be restated in terms of asymptotic models. A space $X$ is asymptotically symmetric if there exists $C$ so that for any infinite array of normalized sequences $(x_j^{(i)})_j$ in $X$ and every permutation $\sigma$ of $\N$  so that both $(x_j^{(i)})_j$ and $(x_j^{(\sigma(i))})_j$ generate asymptotic models $(e_i)_i$ and $(d_i)_i$ respectively we have that $(d_i)_i$ is $C$-equivalent to $(e_{\sigma(i)})_i$. A similar characterization can be given by using the notion of joint spreading models form \cite{AGLM} instead of asymptotic models. Regarding the asymptotic model structure of $\X$, every asymptotic model generated by an array of weakly null sequences in $\X$ is a sequence of a certain type in the space $c_0\oplus\ell_1$.

At the time that this paper was being prepared another Banach space $X_\mathrm{iw}$ from \cite{AM3}, constructed with the purpose of solving a different question, was observed to be asymptotically symmetric without asymptotic-$\ell_p$ or asymptotic-$c_0$ subspaces. This space solves a question of Odell from \cite{O1}, \cite{O}, and \cite{JKO} as to whether every Banach space that admits a uniformly unique spreading model  must have an asymptotic $\ell_p$ or $c_0$ subspace. The space $X_\mathrm{iw}$ is reflexive and it has the property that every spreading model generated by a normalized weakly null sequence is $4$-equivalent to the unit vector basis of $\ell_1$. Our example additionally demonstrates that asymptotically symmetric spaces can lack homogeneity of spreading models in all subspaces.

In Section \ref{defsec} we introduce the necessary definitions and then we define the space $\X$. In Section \ref{propsec} we prove the properties of the space $\X$, namely that it is asymptotically symmetric and that it does not contain a subspace that is asymptotic-$\ell_p$ or asymptotic-$c_0$. We also classify (up to a constant) all the spreading models and asymptotic models admitted by the subspaces of $\X$. Although some results have been proved elsewhere we include all necessary arguments for the sake of self-containment.

\section{Definition of the space $\X$}
\label{defsec}

We define a small variation of the definition of the space $\mathfrak{X}_{0,1}$ from \cite{ABM}. The difference is that we use a coefficient $1/2$ when defining functionals that result from adding very fast growing sequences of averages in the norming set. This gives us better control for estimating a crucial upper estimate (see Proposition \ref{asymptotic model norm estimate}) that will eventually yield the desired result.

\subsection{Preliminaries} 
For two subsets $A$ and $B$ of $\N$ we say $A < B$ if $\max(A) < \min(B)$. We use the convention $\max(\emptyset)=0$ and $\min(\emptyset) = \infty$. For a Banach space $X$ with a Schauder basis $(x_i)_i$ we define the support of a vector $x = \sum_ia_ix_i$ to be the set $\supp(x) = \{i:a_i\neq 0\}$ and we define the range of $x$ to be the smallest interval of $\N$ containing $\supp(x)$. For a vector $x = \sum_{i}a_ix_i$ with finite support and a set $E\subset\mathbb{N}$ we define $Ex = \sum_{i\in E}a_ix_i$. For two vectors $x$ and $y$ in $X$ we write $x<y$ to mean $\supp(x)<\supp(y)$. A finite or infinite sequence $(y_i)_i$ in $X$ is called a block sequence if for all $i>1$ we have $y_{i-1} < y_i$. The space of all scalar sequences with finitely many non-zero entries is denoted by $c_{00}(\N)$ and its unit vector basis is denoted by $(e_i)_i$. Given two elements $f$ and $x$ of $c_{00}(\N)$ we write $f(x)$ to mean the usual inner product on this vector space.

To define the Banach space $\X$ we will first construct an appropriate subset $\W$ of $c_{00}(\N)$, called a norming set. We then consider a norm $\|\cdot\|$ on $c_{00}(\N)$ given by $\|x\| = \sup\{f(x): f\in\W\}$. The space $\X$ will be the completion of  $(c_{00}(\N),\|\cdot\|)$. The following notions are required to define the set $\W$.
\begin{ntt}
Let $G\subset c_{00}(\N)$.
\begin{itemize}

\item[(i)] A vector $\al_0\in c_{00}(\N)$ will be called an $\al$-average of $G$ if there are $d, n\in\N$, with $d\leq n$, and $f_1<\cdots<f_d$ in $G$ so that $\al_0 = (1/n)(f_1+\cdots+f_d)$. We define the size of this $\al$-average to be $s(\al_0) = n$.

\item[(ii)] A finite sequence $(a_i)_{i=1}^k$ of $\al$-averages of $G$ is called admissible if $a_1<\cdots<a_k$ and $k\leq \min\supp(\al_1)$.

\item[(iii)] A finite (or infinite) sequence $(a_i)_i$ of $\al$-averages of $G$ is called very fast growing if $\al_1<\al_2<\cdots$, $s(\al_1)<s(\al_2)<\cdots$, and $s(\al_i) > \max\supp(\al_{i-1})$ for $i>1$.

\item[(iv)] A vector $f$ in $c_{00}(\N)$ will be called a Schreier functional of $G$ if there is an admissible and very fast growing sequence of $\al$-averages of $G$ $(a_i)_{i=1}^k$ so that $f = (1/2)(a_1+\cdots +a_k)$.

\item[(v)] For every Schreier functional $f\in G$ with $f = (1/2)(\al_1+\cdots +\al_k)$ we define the size of $f$ to be $s(f) = s(\al_1)$ and the length of $f$ to be $\ell(f) = k$. A finite (or infinite) sequence $(f_i)_i$ of $\al$-averages of $G$ is called very fast growing if $f_1<f_2<\cdots$, $s(f_1)<s(f_2)<\cdots$, and $s(f_i) > \max\supp(f_{i-1})$ for $i>1$.

\end{itemize}
\end{ntt}
Although the notions of size and length are not necessarily uniquely defined this causes no problems. The notation introduced in item (v) is not necessary to define the space $\X$ we require it however in the proof of the main result.

\subsection{The space $\X$}
We now define the space $\X$ and give an explicit description of the functionals in the norming set $\W$.
\begin{defn}
We define $\W$ to be the smallest symmetric subset $W$ of $c_{00}(\N)$ that contains the unit vector basis of $c_{00}(\N)$, every $\al$-average of $W$, and every Schreier functional of $W$. We define a norm on $c_{00}(\N)$ given by $\|x\| = \sup\{f(x): f\in \W\}$ and we set $\X$ to be the completion of $(c_{00}(\N), \|\cdot\|)$.
\end{defn}

\begin{rem}
\label{inductive construction}
The set $\W$ can be explicitly described by taking the increasing union of a sequence of sets $(W^m_{0,1})_{m=0}^\infty$ where $W^0_{0,1} = \{\pm e_i: i\in\N\}$ and
 \begin{equation*}
 \begin{split}
W_{0,1}^{m+1} =& W_{0,1}^m\cup\Bigg\{\al_0:~\al_0\text{ is an }\al\text{-average of }W_{0,1}^m\Bigg\}\\
&\cup\Bigg\{f:~f\text{ is a Schreier functional of }W_{0,1}^m\Bigg\}.
 \end{split}
 \end{equation*}
 This description of the norming set is fundamental in proving estimates of functionals on vectors.
\end{rem}

\begin{rem}
One can verify by induction on $\N$ that if $f = \sum_ia_ie_i\in\W$ then for any $E\subset \N$ and choice of signs $(\e_i)_i$ the vectors $Ef$ and $\sum_i\e_ia_ie_i$ are both in $\W$. Hence, for any vector $x = \sum_ib_ie_i$ in $\X$ and choice of signs $(\e_i)_i$ we have $\|\sum_ib_ie_i\| = \|\sum_i\e_ib_ie_i\|$, i.e., the basis $(e_i)_i$ of $\X$ is 1-unconditional.
\end{rem}

\begin{rem}
\label{adding schreier}
Let $(f_i)_{i=1}^k$ be a very fast growing sequence of Schreier functionals so that  $\ell(f_1)+\cdots+\ell(f_k)\leq \min\supp(f_1)$. Then $f = f_1+\cdots+f_k$ is in $\W$. This almost trivial observation is important in this paper. It allows us to quantify the $\al$-index and use it to make the necessary estimates (Proposition \ref{asymptotic model norm estimate}).
\end{rem}

\section{Properties of the space $\X$}
\label{propsec}
In this section we define the quantified $\al$-index and use it a tool to give a more precise description of the spreading models and the asymptotic models of the space $\X$ than was possible with the classical $\al$-index.

\subsection{The quantified $\al$-index}
In the majority of constructions that have been performed with the method of saturation under constraints the $\al$-index has been one of the most important tools for describing the spreading models of the corresponding space. The $\al$-index $\al(x_i)_i$ of a block sequence $(x_i)_i$ can take two possible values: zero and not zero. In this paper we assign to a block sequence $(x_i)_i$ a quantified $\al$-index $\tilde\al(x_i)_i$ which is a non-negative real number. Importantly, $\al(x_i)_i$ is zero if and only if $\tilde\al(x_i)_i$ is zero. The actual value of $\tilde\al(x_i)_i$ gives information regarding the spreading models and the asymptotic models of the space $\X$. Let us first recall the definition of the $\al$-index.

\begin{defn}[Definition 3.1 \cite{ABM}]
Let $(x_i)_i$ be a bounded block sequence in $\X$. We define the $\al$-index of $(x_i)_i$ as follows: if for every sequence of very fast growing average $(a_j)_j$ in $\W$ and every subsequence $(x_{i_j})_j$ of $(x_i)_i$ we have $\lim|\al_j(x_{i_j})| = 0$ then we say $\al(x_i)_i = 0$. Otherwise we say that $\al(x_i)_i > 0$.
\end{defn}

The $\tilde\al$-index is an extension of the above definition.

\begin{defn}
\label{quantal}
Let $(x_i)_i$ be a bounded block sequence in $\X$. We define the quantified $\al$-index of $(x_i)_i$ to be the infimum of all $\theta>0$ that have the following property: for all $N\in\N$ there exist $s_0,i_0\in\N$ so that for all Schreier functionals $f\in\W$ with $s(f)\geq s_0$ and $\ell(f)\leq N$ and for all $i\geq i_0$ we have $|f(x_i)| <\theta$.
\end{defn}

Clearly, $0\leq \tilde\al(x_i)_i\leq \limsup_i\|x_i\|$.  The proof of the following is fairly straight forward and it uses the fact that $\W$ is closed under taking restrictions to intervals of $\N$. We include a description of the argument for completeness.

\begin{prop}
Let $(x_i)_i$ be a bounded block sequence in $\X$. Then $\al(x_i)_i = 0$ if and only if $\tilde \al(x_i) = 0$.
\end{prop}

\begin{proof}
Assume that $\tilde \al(x_i)_i = 0$. For $\theta >0$ apply Definition \ref{quantal} for $N=1$. It easily follows that for any very fast growing sequence of $\al$-averages $(\al_j)_j$ in $\W$ and every subsequence $(x_{i_j})_j$ of $(x_i)_i$ we have $\limsup|\al_j(x_{i_j})| \leq 2\theta$. Assume now that $\tilde\al(x_i)_i > 0$, i.e., there exist $\theta > 0$ and $N_0\in\mathbb{N}$ so that for every $s_0,i_0\in\N$ there are $i\geq i_0$ and a Schreier functional $f\in\W$ with $s(f)\geq s_0$ and $l(f)\leq N_0$ so that $|f(x_i)| \geq \theta$. If we write $f$ in the form $f = (1/2)(\al_1+\cdots +\al_k)$, with $k\leq N_0$ and $s(a_q) \geq s_0$ for $1\leq q\leq k$ then there must be an index $q$ so that $|\al_q(x_i)| \geq 2\theta/N_0$. By restricting the range of $\al_q$ we may assume that $\ran(\al_q)\subset \ran(x_i)$. We have thus shown that for every $s_0,i_0\in\N$ there are $i\geq i_0$ and an $\al$-average $\al_0$ in $\W$ with $s(\al_0) \geq s_0$ and $\ran(\al_0)\subset \ran (x_i)$ so that $|\al_0(x_i)| \geq 2\theta/N_0$. It is now straightforward to find a very fast growing sequence of $\al$-averages $(\al_j)_j$ in $\W$ and a subsequence $(x_{i_j})_j$ of $(x_i)_i$ with $\liminf |\al_j(x_{i_j})| \geq 2\theta/N_0$.
\end{proof}

\subsection{Arrays of sequences in $\X$}
The following Proposition provides the main estimate of this paper. It is used to derive estimates for asymptotic models in the space, spreading models in the space, and in the end to prove that the space is asymptotically symmetric.
\begin{prop}
\label{asymptotic model norm estimate}
Let $x_0\in\X$ and $(x_j^{(1)})_i,\ldots,(x_j^{(n)})_i$ be bounded block sequences in $\X$. For every $\e>0$ there exist $j_1<j_2<\cdots<j_n$ so that if we set $x_{j_0}^{(0)} = x_0$ and $A = \|\sum_{i=0}^nx^{(i)}_{j_i}\|$ then
\begin{equation*}
\max\Bigg\{\max_{0\leq i\leq n}\|x^{(i)}_{q_i}\|,\sum_{i=1}^n\tilde\al(x^{(i)}_{j_i})\Bigg\}-\varepsilon\leq A\leq 2\max_{0\leq i\leq n}\|x^{(i)}_{q_i}\| + 2\sum_{i=1}^n\tilde\al(x^{(i)}_{j})+ \varepsilon.
\end{equation*}
\end{prop}

\begin{proof}
As we are allowed a small error $\e>0$ in our estimate we may assume that $x_0$ is finitely supported. For $1\leq i\leq n$ define $\theta_i = \tilde\al(x_j^{(i)}) -\e/n$. Using the definition of the quantified $\al$-index we can find infinite sets $L_1 = \{j_q^{(1)}:q\in\N\},\ldots,L_n = \{j_q^{(n)}:q\in\N\}$, natural numbers $N_1,\ldots,N_n$, and very fast growing sequences of Schreier functionals $(f_q^{(1)})_q,\ldots,(f_q^{(n)})_q$ so that for $1\leq i\leq n$ and $q\in\N$ we have $\ell(f_q^{(i)})\leq N_i$ and $f_q^{(i)}(x_{j_q}^{(i)}) \geq \theta_i$. We may also naturally assume that $\supp(f_q^{(i)}) \subset \supp(x_{j_q}^{(i)})$ and, by perhaps passing to subsequences, we may assume that for all $q_1<\cdots<q_n$ in $\N$ we have that
\begin{gather*}
\supp(x_0)<\supp(x^{(1)}_{q_1})<\cdots<\supp(x^{(n)}_{q_n})\text{ and }\\N_1+\cdots +N_n\leq\min\supp(x^{(1)}_{q_1}).
\end{gather*}
It follows by Remark \ref{adding schreier} that if we pick any $q_1<\cdots<q_n$  then we have that $f = f_{q_1}^{(1)} + \cdots +f_{q_n}^{(n)}$ is in $\W$ and
\begin{equation*}
f\Big(x_0 + \sum_{i=1}^n x^{(i)}_{j_{q_i}}\Big) \geq \sum_{i=1}^n\tilde\al(x^{(i)}_{j_{q_i}}) - \e.
\end{equation*}
It easily follows that for any such $q_1<\cdots<q_n$ we have
\[\max\Bigg\{\max_{0\leq i\leq n}\|x^{(i)}_{q_i}\|,\sum_{i=1}^n\tilde\al(x^{(i)}_{j_i})\Bigg\}-\varepsilon\leq A.\]

We now set out to find $q_1<\cdots<q_n$ so that the desired upper inequality will be satisfied as well. We will choose $q_1<\cdots<q_n$ so that for $1\leq i\leq n$ and $0\leq i'<i$ so that if $N_{i'} = \max\supp(x^{(i')}_{q_{i'}})$ then for every Schreier functional $f\in\W$ with $w(f)\geq \min\supp(x^{(i'+1)}_{q_{i'+1}})$ and $\ell(f) \leq N_i$ we have
\begin{equation}
\label{important inequality}
\Big|f(x^{(i)}_{q_i})\Big| < \tilde \al(x^{(i)}_j)_j+\frac{\e}{2n}.
\end{equation}
We will use the definition of the quantified $\al$-index. Set $N_0 = \max\supp(x_{q_0}^{(0)})$ and for $1\leq i\leq n$ pick $s^0_i,q^0_i\in\N$ so that for every Schreier functional $f\in\W$ with $s(f)\geq s^0_i$ and $\ell(f)\leq N_0$  for all $q\geq q^0_i$ we have that $|f(x^{(i)}_q)| <  \tilde \al(x^{(i)}_j)_j + \e/(2n)$. Pick $q_1$ with $q_1\geq\max_{1\leq i\leq n}q_i^0$ and $\min\supp(x^{(1)}_{q_1})\geq\max_{1\leq i\leq n}s_i^0$. Define $N_1 = \max\supp(x^{(1)}_{q_1})$ and for $2\leq i\leq n$ pick $s^1_i,q^1_i\in\N$ so that for every Schreier functional $f\in\W$ with $s(f)\geq s^1_i$ and $\ell(f)\leq N_1$  for all $q\geq q^1_i$ we have that $|f(x^{(i)}_q)| <  \tilde \al(x^{(i)}_j)_j + \e/(2n)$. Pick $q_2 > q_1$ with  $q_2 \geq \max_{2\leq i\leq n}q_i^1$ and $\min\supp(x^{(2)}_{q_2})\geq\max_{2\leq i\leq n}s_i^1$. Proceed like so.

Define $C = 2\max_{0\leq i\leq n}\|x^{(i)}_{q_i}\| + 2\sum_{i=1}^n\tilde\al(x^{(i)}_j) + \e$. We will prove by induction on $m\in\N$ that for all $f\in\W^m$ (see Remark \ref{inductive construction}) we have $|f(\sum_{i=0}^nx^{(i)}_{q_i})| \leq C$. This is trivial for the case $m = 0$. Assume now that this conclusion holds for every $f\in\W^m$ and let $f\in\W^{m+1}$. If $f$ is an $\al$-average of $\W^m$ then this follows by convexity. Otherwise $f$ is a Schreier functional of $\W^m$ and it may be written as $f = (1/2)\sum_{r = 1}^d\al_r$ where $(\al_r)_{r=1}^d$ is a very fast growing and admissible sequence of $\al$-averages of $\W^m$. We define
\[i_0 = \min\{0\leq i\leq n:\max\supp(f)\geq\min\supp(x^{(i)}_{q_i})\}\]
and
\[r_0 = \max\{1\leq r\leq d:s(\al_r)\leq \min\supp(x^{(i_0+1)}_{q_{i_0+1}})\}.\]
It follows that if we set $g = (1/2)\sum_{r>r_0}\al_r$ then $g$ is a Schreier functional in $\W$ with $s(g) > \min\supp(x^{(i_0+1)}_{q_{i_0+1}})$ and $\ell(g)\leq N_{i_0}$. That is, for $i>i_0$ and the functional $g$ \eqref{important inequality} is satisfied.

We observe that $\max\supp(\al_{r_0-1})<\min\supp(x^{(i_0+1)}_{q_{i_0+1}})$ which yields:
\begin{align*}
\Big|f\Big(\sum_{i=0}^nx^{(i)}_{q_i}\Big)\Big| &\leq |f(x_{q_{i_0}}^{i_0})| + \Big|\Big(\sum_{i>i_0}x^{(i)}_{q_i}\Big)\Big|\\
&\leq \|x_{q_{i_0}}^{(i_0)}\| + \Big|\frac{1}{2}\al_{r_0}\Big(\sum_{i>i_0}x_{q_i}^{(i)}\Big)\Big| + \Big|g\Big(\sum_{i>i_0}x_{q_i}^{(i)}\Big)\Big|\\
&\leq \max_{0\leq i\leq n}\|x_{q_{i}}^{(i)}\| + \frac{1}{2}C + \sum_{i=1}^n\tilde\al(x^{(i)}_j) + \frac{\e}{2} = \frac{1}{2}C + \frac{1}{2}C = C.
\end{align*}
The proof is complete.
\end{proof}

We can now understand, up to an equivalence constant $4$, all asymptotic models of arrays of weakly null sequences in the space $\X$. In fact, they are all certain sequences in $c_0\oplus\ell_1$.

\begin{cor}
\label{asmod char}
Let $(x_j^{(i)})_j$ be an infinite array of normalized weakly null sequences in $X$ that generate and asymptotic model $(z_i)_i$. Then there exist a sequence of non-negative scalars $(w_i)_i$ so that for any $n\in\N$ and sequence of scalars $(\la_i)_{i=1}^n$ we have
\[ \max\Bigg\{\max_{1\leq i\leq n}|\la_i|, \sum_{i=1}^nw_i|\la_i|\Bigg\}\leq \Bigg\|\sum_{i=1}^n\la_iz_i\Bigg\| \leq 2\max_{1\leq i\leq n}|\la_i| + 2\sum_{i=1}^nw_i|\la_i|.\]
In particular, $(z_i)_i$ is $4$-equivalent to the sequence $(e_i,w_ie_i)_i$ in $(c_0\oplus\ell_1)_\infty$.
\end{cor}
\begin{proof}
Set $x_0 = 0$ and for $i=1,\ldots,n$ define $(x_j^{(i)})_j = (\la_i x^{(i)}_j)_j$ and apply Proposition \ref{asymptotic model norm estimate} to obtain that $w_i = \tilde\al(x^{(i)}_j)$, $i\in\N$ are the desired scalars.
\end{proof}

\subsection{Sequences in $\X$}
The fact that every spreading model generated by a weakly null sequence  in $\X$ is equivalent to either the unit vector basis of $c_0$ or of $\ell_1$ and that every subspace of $\X$ admits both of these spreading models is proved in a nearly identical manner as it was proved in \cite{ABM}. The idea is the following: a sequence $(x_i)_i$ generating a $c_0$ spreading model can be blocked by setting $y_n = \sum_{i\in F_n}x_i$ appropriately so that $(y_n)_n$ generates an $\ell_1$ spreading model. Similarly, a sequence $(x_i)_i$ generating an $\ell_1$ spreading model can be blocked by setting $y_n = (1/\#F_n)\sum_{i\in F_n}x_i$ appropriately so that $(y_n)_n$ generates an $\ell_1$ spreading model. For the sake of self-containment we include the proof.

The following states that that every spreading model of a weakly null sequence in $\X$ is either equivalent to the unit vector basis of $c_0$ or to the unit vector basis of $\ell_1$. This was proved in a slightly different manner in \cite{ABM}. Here the result follows almost immediately from  Proposition \ref{asymptotic model norm estimate}.

\begin{cor}
Let $(x_j)_j$ be a normalized bock sequence in $\X$ and assume that it generates some spreading model $(e_i)_i$. Let $\al = \tilde\al(x_j)$. Then for any $n\in\N$ and scalars $(\la_i)_{i=1}^n$ we have
\[ \max\Bigg\{\max_{1\leq i\leq n}|\la_i|, \al\sum_{i=1}^n|\la_i|\Bigg\}\leq \Bigg\|\sum_{i=1}^n\la_ie_i\Bigg\| \leq 2\max_{1\leq i\leq n}|\la_i| + 2\al\sum_{i=1}^n|\la_i|.\]
In particular, if $\tilde\al(x_i) = 0$ then $(e_i)_i$ is equivalent to the unit vector basis of $c_0$ and otherwise it is equivalent to the unit vector basis of $\ell_1$.
\end{cor}

\begin{proof}
Set $x_0 = 0$ and for $i=1,\ldots,n$ define $(x_j^{(i)})_j = (\la x_j)_j$ and apply Proposition \ref{asymptotic model norm estimate}.
\end{proof}

We next intend to prove that both $c_0$ and $\ell_1$ appear as spreading models in every subspace. The following lemma is well known but we include a proof for completeness.

\begin{lem}
\label{average on average}
Let $x_1<\cdots<x_n$ be normalized finitely supported vectors in $\X$. Then for any  $\al$-average $\al_0$ in $\X$ we have that
\[\Big|\al_0\Big(\frac{1}{n}\sum_{i=1}^nx_i\Big) \Big|\leq \frac{1}{s(\al_0)} + \frac{2}{n}.\]
\end{lem}

\begin{proof}
Let $\al_0 = (1/d)(f_1+\cdots+f_k)$ where $f_1<\cdots<f_k$ are in $\W$ and $k\leq d$. Define $A = \{i: \ran(x_i)\cap\ran(f_j)\neq\emptyset$ for at most one $j\}$. Then for $i\in A$ we have $|\al_0(x_i)| \leq 1/d$. For $i\not\in A$ define the set $F_i = \{j:\ran(x_i)\cap\ran(f_j)\neq\emptyset\}$. It follows that $\max(F_i)\leq \min(F_{i'})$ for all $i<i'\not\in A$ and therefore $\sum_{i\not\in A}\#F_i \leq 2k$. We condlude:
\[\Big|\al_0\Big(\frac{1}{n}\sum_{i=1}^nx_i\Big) \Big|\leq \frac{1}{n}\sum_{i\in A}|\al_0(x_i)| + \frac{1}{n}\sum_{i\not\in A}\frac{\#F_k}{d} \leq \frac{1}{d}+\frac{2}{n}.\]
\end{proof}

\begin{prop}
\label{its all there}
Let $X$ be a block subspace of $\X$. Then there exists a normalized block sequence in $X$ that generates a spreading model equivalent to the unit vector basis of $\ell_1$ and there exists another normalized block sequence in $X$ that generates a spreading model equivalent to the unit vector basis of $c_0$.
\end{prop}

\begin{proof}
Start with an arbitrary normalized block sequence $(x_i)_i$ in $X$ that generates some spreading model $(e_i)_i$. Pick for each $i\in\N$ an $f_i\in\W$ with $f_i(x_i) = 1$ and $\ran(f_i) \subset \ran(x_i)$. Choose successive subsets of the natural numbers $(F_n)_n$ with $\#F_n\to\infty$ and $\#F_n\leq\min(F_n)$. If $(e_i)_i$ is equivalent to the unit vector basis of $c_0$  set $y_n = \sum_{i\in F_n}x_i$ and $\al_n = (1/\#F_n)\sum_{i\in F_n}f_i$. It follows that there is $C>0$ so that $\sup\|y_n\| \leq C$ and for all $n\in\N$ $|\al_n(y_n)|\geq 1$. Thus $(y_n)_n$ is bounded and it has positive $\al$-index, i.e., it has a subsequence generating an $\ell_1$ spreading model. If on the other hand $(e_i)_i$ is equivalent to the unit vector basis of $\ell_1$ set $y_n = (1/\#F_n)\sum_{i\in F_n}x_i$. Then there exists $c>0$ so that $\inf\|y_n\| \geq c$ and by Lemma \ref{average on average} we the $\al$-index of $(y_n)_n$ is zero, i.e., it has a subsequence generating a $c_0$ spreading model.
\end{proof}

Since $\X$ has an unconditional subspace, and by Proposition \ref{its all there} it has no subspace isomorphic to $c_0$ or to $\ell_1$, we conclude the following by James' theorem \cite{J}.
\begin{cor}
The space $\X$ is reflexive.
\end{cor}

\begin{rem}
We observed that every asymptotic model generated by an array of weakly null sequences in $\X$ is $4$-equivalent to a sequence of the form $(e_i,w_ie_i)_i$ in $(c_0\oplus\ell_1)_\infty$. A converse of this is also true: in every infinite dimensional subspace $X$ of $\X$ and every sequence $(w_i)_i$ in $[0,1]$ there exists an array of normalized weakly null sequences in $X$ that generate an asymptotic model $20$-equivalent to the sequence $(e_i,w_ie_i)_i$ in $(c_0\oplus\ell_1)_\infty$. The way to achieve this is to take a normalized weakly null sequence $(x_j)_j$ in $X$ that generates a $c_0$ spreading model and a normalized weakly null sequence $(y_n)_n$ that generates an $\ell_1$ spreading model with constant $5/4$. Assuming that for all $j\in\N$ we have $x_j<y_j<x_{j+1}$ define for each $i,j\in\N$ the vector $z^{(i)}_j =\|x_j + w_iy_j\|^{-1} (x_j + w_iy_j)$. It can be seen that the sequences $(z^{(i)}_j)_j$ satisfy $w_i/5\leq\tilde\al(z^{(i)}_j)_j\leq w_i$ and hence by Corollary \ref{asmod char} any asymptotic model generated by a sub-array of $(z^{(i)}_j)_j$, $i\in\N$ must be $20$-equivalent to $(e_i,w_ie_i)_i$ in $(c_0\oplus\ell_1)_\infty$.
\end{rem}

\subsection{Conclusion}
We now put all the pieces together to show that the space is asymptotically symmetric, despite not having a unique spreading model in any subspace.

\begin{thm}
The space $\X$ is asymptotically symmetric.
\end{thm}

\begin{proof}
Let $(x_j^{(i)})_j$, $1\leq i\leq n$ be an array of bounded sequences in $\X$, $\sigma$ be a permutation of $\{1,\ldots,n\}$ and assume that that the limits
\begin{equation*}
A = \lim_{j_1\to\infty}\cdots\lim_{j_n\to\infty}\Big\|\sum_{i=1}^nx^{(i)}_{j_i}\Big\|\text{ and }
B =  \lim_{j_1\to\infty}\cdots\lim_{j_n\to\infty}\Big\|\sum_{i=1}^nx^{(\sigma(i))}_{j_i}\Big\|
\end{equation*}
both exist. By reflexivity and passing to subsequences we may assume that the limits $w$-$\lim_jx_j^{(i)} = x_i$, $1\leq i\leq n$ exist. Define $y_0 = \sum_{i=1}^nx_i$ and $\la_0 = \|y_0\|$. We may also assume that the sequences $(y^{(i)}_j)_j = (x_j^{(i)}-x_i)_j$ are block sequences and that the numbers $\la_i = \lim_j\|y^{(i)}_j\|$ exist for $1\leq i\leq n$. Note that
\[A = \lim_{j_1\to\infty}\cdots\lim_{j_n\to\infty}\Bigg\|y_0 + \sum_{i=1}^ny^{(i)}_{j_i}\Bigg\| \text{ and } B = \lim_{j_1\to\infty}\cdots\lim_{j_n\to\infty}\Bigg\|y_0 + \sum_{i=1}^ny^{(\sigma(i))}_{j_i}\Bigg\|.\]
Proposition \ref{asymptotic model norm estimate} yields that
\[\max\Big\{\max_{0\leq i\leq n}\la_i,\sum_{i=1}^n\tilde\al(y^{(i)}_j)\Big\}\leq A\leq 2\max_{0\leq i\leq n}\la_i + 2\sum_{i=1}^n\tilde\al(y^{(i)}_j)\]
and the exact same estimate for $B$ instead of $A$. This means $A\leq 4B$.
\end{proof}

It was proved in \cite{OS1} that there exist Banach spaces that do not admit an $\ell_p$ or $c_0$ spreading model. Although asymptotically symmetric Banach spaces do not necessarily have a unique spreading model a possible implication of this property could perhaps be the existence of an $\ell_p$ or $c_0$ spreading model. The following can be viewed as necessary modification of Problem \ref{main question}.
\begin{problem}
Does every asymptotically symmetric Banach spaces admit an $\ell_p$ or $c_0$ spreading model?
\end{problem}

\end{document}